\documentclass[reqno,12pt]{amsart}
\usepackage[utf8]{inputenc}
\usepackage[left=1.2in, right=1.2in, top=1in]{geometry}

\usepackage{amsmath, amssymb,amsthm, mathrsfs,color,times,textcomp,verbatim}
\usepackage{xcolor}
\usepackage[colorlinks=true]{hyperref}
\hypersetup{urlcolor=blue, citecolor=red, linkcolor=blue}
\usepackage[square,numbers]{natbib}
\usepackage{pgfplots}
\usepackage{tikz}
\usepackage{fancyhdr}
\usepackage{lipsum}
\numberwithin{equation}{section}
\theoremstyle{plain}
\newtheorem{theorem}{Theorem}[section]

\newtheorem{lemma}[theorem]{Lemma}

\usepackage{graphicx}

\theoremstyle{definition}

\def\R{\mathbb R}

\def\O{\Omega}
\def\p{\partial}

\def\a{\alpha}

\def\s{\sigma}

\def\non{\nonumber}
\def\rz{\mathbb{R}}

\newtheorem{remark}[theorem]{Remark}

\allowdisplaybreaks[4]

\title[Liouville Theorem for $(p,q)$-Laplace Equations]{Liouville Theorem for $(p,q)$-Laplace Equations in the Subcritical Case}
\begin{document}
\author{Yang Zhou}
\address{School of Mathematical Sciences, University of Science and Technology of China, Hefei, Anhui Province, P. R. China, 230026}
\email{zy19700816@mail.ustc.edu.cn}
\author{Hua Zhu}
\address{School of Mathematics and Physics, Southwest University of Science and Technology, Mianyang 621010, SiChuan PROVINCE, P. R. China}
\email{zhuhmaths@mail.ustc.edu.cn}
\maketitle

\begin{abstract}
    We employ the vector field method to establish a Liouville-type theorem for a class of \((p,q)\)-Laplace equations in the Euclidean space \(\mathbb{R}^n\). By modifying the exponents in the differential identity, we prove nonexistence in the subcritical range \(p-1<\alpha<q^*-1\), where \(q^*=nq/(n-q)\). The approach relies on constructing a suitable differential identity, carrying out precise integral estimates with cutoff functions, and combining sign control and decay of the cutoff errors.

	\vskip0.3cm
	\noindent{\bfseries Keywords:}{ Liouville theorem, $(p,q)$-Laplace operator, double phase, vector field method.}\\
\end{abstract}
\section{Introduction}
Liouville-type theorems are a central topic in the research of partial differential equations, as they address the fundamental question of the nonexistence of nontrivial solutions to nonlinear elliptic equations. Such results play an indispensable role in the analysis of a priori estimates, existence theory, and symmetry properties of solutions. Over the past several decades, Liouville theorems have been extensively investigated across all these settings, leading to the development of powerful tools such as the method of moving planes, the moving spheres technique and vector field methods.

For the subcritical semilinear elliptic equation
\begin{equation}\label{eq:elliptic}
	\Delta u + u^q = 0 \quad \text{in } \mathbb{R}^n.
\end{equation}
Gidas–Spruck \cite{GS81} established the Obata-type identity \cite{Obata71} via the method of vector fields and showed that equation \eqref{eq:elliptic} has no positive solutions when \(1 < q < 2^*-1\), where $ 2^* = \frac{2n}{n-2}. $

Serrin--Zou \cite{Serrin02} extended this result to the following $p$-Laplacian equation
\begin{equation}\label{eq:1.2}
	\Delta_p u + u^q = 0 \quad \text{in } \mathbb{R}^n.
\end{equation}
They showed similarly that equation \eqref{eq:1.2} admits no positive solutions when $p-1 < q < p^*-1, $ where $ p^* = \frac{np}{n-p}$ denotes the Sobolev critical exponent. A new proof using invariant tensor methods was later given by Zhu \cite{Zhu24}, who also relaxed the condition to $q < p^*-1 $. Based on the invariant tensor technique, Liang-Yan-Wu \cite{LWY2025} obtained the Liouville-type theorem for the anisotropic subcritical $ p $-Laplacian equation in the Euclidean space by means of the vector field method.

The critical case of equation \eqref{eq:elliptic}, namely \(q=2^*-1\), is closely related to the Yamabe problem in geometry. In this case, positive solutions are explicitly given by the Aubin–Talenti bubbles 
 \begin{equation}\label{1}
 	u(x)=\bigg(\frac{\lambda \sqrt{n(n-2)}}{\lambda^2+|x-x_0|^2}\bigg)^{\frac{n-2}{2}},~~ \lambda>0,~~ x_0\in \mathbb{R}^{n}.
 \end{equation}
Under the additional hypothesis $u(x)=O(|x|^{2-n})$, Gidas-Ni-Nirenberg \cite{GNN81} used the moving planes technique to show that the solution of \eqref{eq:elliptic} must be of the form \eqref{1}. Caffarelli-Gidas-Spruck \cite{CGS89} also studied equation \eqref{eq:elliptic} for $ q=2^*-1 $. Without assuming that the solution has appropriate decay estimates, they used the Kelvin transform and the method of moving spheres to show that the solution must be of the form \eqref{1}. Chen-Li \cite{CL91} gave a concise new proof in 1991. In 2003, Chang-Gursky-Yang \cite{CGY03} used a vector field method to show that for the critical case of equation \eqref{eq:elliptic}, any positive solution must be of the form \eqref{1} without any additional conditions when $ n=3 $, and under an integral boundedness condition when $ n\geqslant 4 $. \\

 As for the critical case $q=p^*-1$, under the assumption of finite energy
 \begin{equation}
     u \in D^{1,p}(\mathbb{R}^n) :=\{u\in L^{p^*}(\mathbb{R}^n), Du \in L^p(\mathbb{R}^n) \}.
 \end{equation}
 Many scholars have used asymptotic analysis and the moving plane method to prove that the positive solutions of equation \eqref{eq:1.2} must be of the following form \eqref{form-2}. For example, Damascelli-Merchan-Montoro-Sciunzi \cite{DMMS14} established this result for $\frac{2n}{n+2} < p < 2$, Vétois \cite{Vetois} for $1 < p < 2$, and Sciunzi \cite{Sciunzi} for $2 < p < n$.
 \begin{equation}\label{form-2}
     u(x) = \bigg(\frac{\lambda^{\frac{1}{p-1}}n^{\frac{1}{p}}(\frac{n-p}{p-1})^\frac{p-1}{p}}{|x-x_0|^\frac{p}{p-1}+\lambda^\frac{p}{p-1}}\bigg)^\frac{n-p}{p} ,~~ \lambda>0,~~ x_0\in \mathbb{R}^{n}.
 \end{equation}
It should be noted that the moving-plane method strongly relies on the symmetries of both the domain and the equation, and therefore it is not well suited to anisotropic settings. Relying on the asymptotic estimates from \cite{Vetois,Sciunzi}, Ciraolo-Figalli-Roncoroni \cite{CFR20} employed a vector field approach to classify positive solutions of the critical $ p $-Laplacian equation in the anisotropic setting. The classification problem without finite-energy assumptions remains an important research topic, and partial results \cite{CMR23} \cite{Ou25} \cite{Vé24} have been obtained.

Liouville theorems and the classification of solutions to PDEs in half-spaces have also been systematically studied; see, e.g., \cite{CCFS,E90,Li03,YuZhou25,Zhou24} and the references therein.

In this paper, we will study  the following $ (p,q) $-Laplace equation:
	\begin{equation}\label{pde1}
\Delta_p u+\Delta_q u + u^\alpha =0 , u\geqslant0,  ~~\text{in} ~~\mathbb{R}^{n},
\end{equation}
where $1< q<p$, $\Delta_p u = \operatorname{div}(|\nabla u|^{p-2}\nabla u)$ and $\Delta_q u = \operatorname{div}(|\nabla u|^{q-2}\nabla u)$ are the $p$-Laplacian and $q$-Laplacian, respectively, which together form the $(p,q)$-Laplace operator of unbalanced growth.

The $(p,q)$-Laplace operator is a special case of \textbf{double-phase differential operators}, originally introduced by Zhikov \cite{Z1995} to model the mechanical behavior of strongly anisotropic materials. Furthermore, the $(p,q)$-Laplace operator also appears as a special case of multi-phase differential operators in the Born-Infeld equation (after Taylor expansion) from electromagnetism, electrostatics, and electrodynamics. This class of operators is also used to model physical phenomena such as non-Newtonian fluids and composites composed of materials with different hardening exponents.
There have been a number of papers focusing on Liouville-type theorems for equations with the $(p,q)$-Laplace operator as the principal operator; see References \cite{BBF26,WZ26,Zha26} and the references therein.

This paper aims to establish a Liouville theorem of equation \eqref{pde1} in the subcritical case (see Remark \ref{r2}). We now state our main result.

\begin{theorem}\label{Them1}
	If $1<q<p$, $q<n$, and
	\[
		p-1<\alpha<q^*-1=\frac{n(q-1)+q}{n-q},
		\qquad q^*=\frac{nq}{n-q},
	\]
	then there exists no positive weak solution of \eqref{pde1} in $\mathbb R^n$.
\end{theorem}

\begin{remark}
	The upper bound in Theorem \ref{Them1} improves the Serrin's index range obtained in \cite{BBF26}, namely
	\[
		\alpha<q_*-1,\quad q_*:=\frac{q(n-1)}{n-q}.
	\]
	 
\end{remark}

\begin{remark}\label{r2}
Let $\mathcal{H}:\Omega \times [0,\infty)\to [0,\infty)$ be defined by
\[
(x,t)\mapsto t^p+t^q.
\]
The Luxemburg norm is given by
\[
\|u\|_{\mathcal{H}}
=
\inf
\left\{
\tau>0:\rho_{\mathcal{H}}\left(\frac{u}{\tau}\right):=
\int_{\Omega}\mathcal{H}(x,|u|/\tau)\,dx\leq 1
\right\}.
\]

    The range of exponents in Theorem \ref{Them1} corresponds to the critical exponent in the following Sobolev-type inequality (\cite{CS16}): if $\O\subset \rz^n$ is bounded, then
    $$\|v\|_{L^{t}(\O)}\le C\|\nabla v\|_{\mathcal{H}(\O)}$$
    for any $v\in C_c^\infty(\O)$ and $t\in [1,q^*]$.
\end{remark}

The remainder of the paper is organized as follows. In Section 2, we introduce the notation and preliminary lemmas needed in the proof. Section 3 is devoted to the proof of Theorem \ref{Them1}.

\section{Preliminaries}

Throughout this paper, we use \(B_{r}(x)\) to denote the Euclidean ball in \(\mathbb{R}^{n}\) centered at $ x $ with radius $ r $. For the sake of brevity, we write \(B_{r}=B_{r}(0)\) when the center is at the origin.

To establish the main result of this paper, we recall two critical lemmas from existing literature, which play essential roles in the subsequent analysis.

\begin{lemma}[Lemma 4.5,\cite{AKM18}]\label{lemma:3.5}
	Let the matrix \( A \) be symmetric with positive eigenvalues and let \( \lambda_{\min} \) and \( \lambda_{\max} \) be its smallest and largest eigenvalue, respectively; let \( B \) be a symmetric matrix, then
	\[
	\text{trace}(AB(AB)^T) \leq n \left( \frac{\lambda_{\max}}{\lambda_{\min}} \right)^2 \text{trace}((AB)^2).
	\]
	In particular, \(\text{trace}((AB)^2)\geq0\).
\end{lemma}

\begin{lemma}[Lemma 1.2,\cite{BBF26}]\label{lemma1}
	Suppose \(\{|x|>R>0\}\subset\Omega\). Let \(u\) be a positive weak solution of the inequality
	\[
	\Delta_p u+\Delta_q u \le 0, \quad x\in\Omega.
	\]
	If \(1<q<n\), there exists a positive constant \(C=C(p,q,n,u,R)\) such that
	\[u(x)\geq C|x|^{-\frac{n-q}{q-1}}.\]
	If \(q\ge n\), then\[\liminf_{|x|\to\infty}u(x)>0.\]
\end{lemma}

\section{Proof of the theorem \ref{Them1}}
In this section, we use the vector field method to prove that the equation admits no positive weak solutions under the given conditions, by constructing a differential identity, choosing suitable parameters, performing integral estimates with cutoff functions, and combining sign control and decay of the cutoff errors. Finally, the standard regularization argument for the quasi-linear equation finishes the proof of Theorem \ref{Them1} (see e.g. \cite{CM18,CFR20,Zhou24,YuZhou25}).

\noindent\textbf{Conventions:} the following calculations are carried out in $\O_r:=\{|\nabla u|>0\}\cap \rz^n$. By the standard regularity theorem, we know the weak solution $u$ to equation \eqref{pde1} is indeed smooth in $\O_r$ and belongs to $C_{loc}^{1,\alpha}(\rz^n)$ (see e.g. \cite{LU64,L91}). 

\subsection{Notation and Differential Identity}

Define the vector field associated with the $(p,q)$-Laplace by
\[
Y^{i}:=|\nabla u|^{p-2} u_{i}+|\nabla u|^{q-2} u_{i},
\]
and denote its Jacobian matrix by \(Y_{j}^{i}:=\partial_{j}Y^{i}\). Therefore $ \sum\limits_{i=1}^{n}Y^{i}_{i}=\Delta_{p}u+\Delta_{q}u=-u^{\alpha} $.
Set
\[
	s=|\nabla u|,\qquad x=s^p+s^q,\qquad z=(p-1)s^p+(q-1)s^q.
\]
The corrected proof uses the \(q\)-phase weighted energy
\begin{equation}\label{Xq-def}
	X_q(s):=s^q+\frac{q(p-1)}{p(q-1)}s^p.
\end{equation}
Then
\begin{equation}\label{Xq-key}
	\frac{sX_q'(s)}{(p-1)s^p+(q-1)s^q}=\frac{q}{q-1}.
\end{equation}
We shall also use
\[
	r:=\frac{X_q(s)}{x}\geq1,\qquad h_q:=\frac{q}{q-1}.
\]

\begin{lemma}
For any parameters \(\beta,\lambda,\mu\in\mathbb{R}\), the following differential identity holds:
	\begin{align}\label{eq-1}
		\begin{split}
				&\partial_i \left[ u^\beta Y_j^i Y^j - u^\beta Y_j^j Y^i + \lambda u^{\beta-1} X_qY^i + \mu u^{\beta+\alpha} Y^i \right] \\
			&= u^\beta Y_j^i Y_i^j+\beta u^{\beta-1}u_iY_j^iY^j-(1+\mu)u^{\beta+2\alpha}+\lambda(\beta-1)u^{\beta-2}X_qx\\
			&\quad+\lambda u^{\beta-1}(X_q)_iY^i-\lambda u^{\beta+\alpha-1}X_q+\big[\mu(\beta+\alpha)+\beta\big]u^{\beta+\alpha-1}x,
		\end{split}
	\end{align}
	where \( \beta, \lambda, \mu \) are undetermined parameters.
\end{lemma}
\begin{proof}
	Differentiate the vector field in \eqref{eq-1} and use \(Y_i^i=-u^\alpha\). The terms containing \(\partial_iY_j^i\) and \(\partial_iY_j^j\) cancel, since
	\[
		(\partial_iY_j^i)Y^j-(\partial_iY_j^j)Y^i
		=-\alpha u^{\alpha-1}u_jY^j+\alpha u^{\alpha-1}u_iY^i=0.
	\]
	This gives the displayed identity. In particular, the coefficient of \(u^{\beta+2\alpha}\) is \(-(1+\mu)\).
\end{proof}

To determine the sign of the first term on the right-hand side of equality \eqref{eq-1}, we perform the following analysis.
\begin{align}
	\begin{split}
		Y^{i}_{j}&=\partial_j\left(|\nabla u|^{p-2}u_i + |\nabla u|^{q-2}u_i\right)\\
		&=|\nabla u|^{p-2}\left(u_{ij} + (p-2)\frac{u_k u_{kj} u_i}{|\nabla u|^2}\right)
		+ |\nabla u|^{q-2}\left(u_{ij} + (q-2)\frac{u_k u_{kj} u_i}{|\nabla u|^2}\right) \\
		&:=A\cdot B,
	\end{split}
\end{align}
where
\[
	A=|\nabla u|^{p-2}\left(I+(p-2)\frac{\nabla u\otimes\nabla u}{|\nabla u|^2}\right)
	+|\nabla u|^{q-2}\left(I+(q-2)\frac{\nabla u\otimes\nabla u}{|\nabla u|^2}\right),
	\qquad B=\nabla^2u.
\]
The matrix \(A\) is positive definite. By Lemma \ref{lemma:3.5}, \begin{equation}\label{eq3.5}
    Y_j^iY_i^j=\operatorname{trace}((AB)^2)\geq c|D Y|^2\ge 0,
\end{equation}
where $c>0$ depends only on $n,p,q$.

At the calculation point, we rotate the coordinate system such that $ \nabla u=|\nabla u|e_{1} $, i.e., $u_{1}\neq 0 $ and $u_{i}=0  $ for $ i=2,\dots,n $.
Therefore

$(Y_j^i) = 
\begin{pmatrix}
	\big((p-1)|\nabla u|^{p-2} + (q-1)|\nabla u|^{q-2}\big)u_{11}, & 
	\big((p-1)|\nabla u|^{p-2} + (q-1)|\nabla u|^{q-2}\big)u_{1j} \\[6pt]
	\big(|\nabla u|^{p-2} + |\nabla u|^{q-2}\big)u_{i1}, & 
	\big(|\nabla u|^{p-2} + |\nabla u|^{q-2}\big)u_{ij}
\end{pmatrix},
$
\begin{align*}
    u_iY_j^iY^j=\bigl( |\nabla u|^p + |\nabla u|^q \bigr) \bigl( (p-1)|\nabla u|^{p-2} + (q-1)|\nabla u|^{q-2} \bigr) u_{11},
\end{align*}
and by (\ref{Xq-key})
\begin{align*}
    (X_q)_iY^i=\frac{q}{q-1}\bigl( |\nabla u|^p + |\nabla u|^q \bigr) \bigl( (p-1)|\nabla u|^{p-2} + (q-1)|\nabla u|^{q-2} \bigr) u_{11}.
\end{align*}
Substituting this into the identity \eqref{eq-1} simplifies it to:
	\begin{align}\label{eq-2}
	\begin{split}
		& \partial_i \bigl[ u^\beta Y_j^i Y^j - u^\beta Y_j^j Y^i + \lambda u^{\beta-1} X_q Y^i + \mu u^{\beta+\alpha} Y^{i} \bigr] \\
	=& u^\beta \bigl( (p-1)|\nabla u|^{p-2} + (q-1)|\nabla u|^{q-2} \bigr)^2 u_{11}^2 + u^\beta \bigl( |\nabla u|^{p-2} + |\nabla u|^{q-2} \bigr)^2 \sum_{i,j\ge2} u_{ij}^2  \\
	+& 2u^\beta \bigl( |\nabla u|^{p-2} + |\nabla u|^{q-2} \bigr) \bigl( (p-1)|\nabla u|^{p-2} + (q-1)|\nabla u|^{q-2} \bigr) \sum_{i\ge2} u_{1i}^2\\
	+& (h_q\lambda+\beta) u^{\beta-1} \bigl( |\nabla u|^p + |\nabla u|^q \bigr) \bigl( (p-1)|\nabla u|^{p-2} + (q-1)|\nabla u|^{q-2} \bigr) u_{11} \\
	-& (1+\mu) u^{\beta+2\alpha} 
	+ \lambda(\beta-1) u^{\beta-2} \bigl( |\nabla u|^p + |\nabla u|^q \bigr)\bigl( |\nabla u|^q + \frac{q(p-1)}{p(q-1)}|\nabla u|^p \bigr)
	\\
	+& \bigl( \mu(\beta+\alpha) + \beta \bigr) u^{\alpha+\beta-1} \bigl( |\nabla u|^p + |\nabla u|^q \bigr) -\lambda u^{\alpha+\beta-1} \bigl( |\nabla u|^q + \frac{q(p-1)}{p(q-1)}|\nabla u|^p \bigr).
	\end{split}
\end{align}

From the equation \(\Delta_{p}u+\Delta_{q}u=-u^{\alpha}\), direct expansion gives
\begin{align}\label{eq-3}
		\Delta_p u + \Delta_q u 
	= \bigl( |\nabla u|^{p-2} + |\nabla u|^{q-2} \bigr) \sum_{i=2}^{n} u_{ii} 
	+ \bigl( (p-1) |\nabla u|^{p-2} + (q-1) |\nabla u|^{q-2} \bigr) u_{11}.
\end{align}
Squaring both sides and dividing by $ n $:
\begin{align}
	\begin{split}\label{eq-4}
		\frac{1}{n} u^{2\alpha} 
	&= \frac{1}{n} \bigl( u^\alpha \bigr)^2 
	= \frac{1}{n} \bigl( \Delta_p u + \Delta_q u \bigr)^2 \\
	&= \frac{1}{n} \left( \bigl( |\nabla u|^{p-2} + |\nabla u|^{q-2} \bigr) \sum_{i=2} u_{ii} + \bigl( (p-1)|\nabla u|^{p-2} + (q-1)|\nabla u|^{q-2} \bigr) u_{11} \right)^2.
	\end{split}
\end{align}
For \(i,j\geq2\), we use the orthogonal decomposition:
\begin{align}\label{eq-5}
\begin{split}
   	\sum_{i,j \ge 2} u_{ij}^2 
	&= \sum_{i,j \ge 2} \left( u_{ij} - \frac{1}{n-1} \sum_{k \ge 2} u_{kk} \delta_{ij} + \frac{1}{n-1} \sum_{k \ge 2} u_{kk} \delta_{ij} \right)^2 \\
	&= \sum_{i,j \ge 2} \left( u_{ij} - \frac{1}{n-1} \sum_{k \ge 2} u_{kk} \delta_{ij} \right)^2 
	+ \frac{1}{n-1} \left( \sum_{k \ge 2} u_{kk} \right)^2. 
\end{split}
\end{align}

According to \eqref{eq-4} and \eqref{eq-5}, it follows that the first two terms on the right-hand side of \eqref{eq-2} can be transformed into
	\begin{align}\label{eq-6}
	\begin{split}
		& u^\beta \bigl( (p-1)|\nabla u|^{p-2} + (q-1)|\nabla u|^{q-2} \bigr)^2 u_{11}^2 + u^\beta \bigl( |\nabla u|^{p-2} + |\nabla u|^{q-2} \bigr)^2 \sum_{i,j\ge2} u_{ij}^2  \\
		&= \frac{1}{n} u^{\beta+2\alpha}
		+ \frac{n-1}{n} u^\beta \bigl( (p-1)|\nabla u|^{p-2} + (q-1)|\nabla u|^{q-2} \bigr)^2 u_{11}^2 \\
		&\quad + \frac{1}{n(n-1)} u^\beta \bigl( |\nabla u|^{p-2} + |\nabla u|^{q-2} \bigr)^2 \left( \sum_{k \ge 2} u_{kk} \right)^2 \\
		&\quad - \frac{2}{n} u^\beta \bigl( |\nabla u|^{p-2} + |\nabla u|^{q-2} \bigr) \sum_{i \ge 2} u_{ii}
		\cdot \bigl( (p-1)|\nabla u|^{p-2} + (q-1)|\nabla u|^{q-2} \bigr) u_{11} \\
		&\quad + u^\beta \bigl( |\nabla u|^{p-2} + |\nabla u|^{q-2} \bigr)^2 \sum_{i,j \ge 2} \left( u_{ij} - \frac{1}{n-1} \sum_{k \ge 2} u_{kk} \delta_{ij} \right)^2\\
		&= \frac{n-1}{n} u^\beta \Bigl[
		\bigl( (p-1)|\nabla u|^{p-2} + (q-1)|\nabla u|^{q-2} \bigr) u_{11}
		- \frac{1}{n-1} \bigl( |\nabla u|^{p-2} + |\nabla u|^{q-2} \bigr) \sum_{k \ge 2} u_{kk}
		\Bigr]^2 \\
		& \quad +\frac{1}{n} u^{\beta+2\alpha}
		 + u^\beta \bigl( |\nabla u|^{p-2} + |\nabla u|^{q-2} \bigr)^2
		\sum_{i,j \ge 2} \left( u_{ij} - \frac{1}{n-1} \sum_{k \ge 2} u_{kk} \delta_{ij} \right)^2.
	\end{split}
\end{align}

By applying \eqref{eq-6} again in \eqref{eq-4} and then substituting \eqref{eq-2}, we further reduce the identity \eqref{eq-1} to
\begin{align}\label{eq-7}
	\begin{split}
		&\partial_i \Bigl[ u^\beta Y_j^i Y^j - u^\beta Y_j^j Y^i + \lambda u^{\beta-1} X_q Y^i + \mu u^{\beta+\alpha} Y^{i} \Bigr]\\
	&= \frac{n-1}{n} u^\beta \left[ \frac{n}{n-1} \bigl( (p-1)|\nabla u|^{p-2} + (q-1)|\nabla u|^{q-2} \bigr) u_{11} + \frac{1}{n-1} u^\alpha \right]^2 \\
	& + (h_q\lambda+\beta) u^{\beta-1} \bigl( |\nabla u|^p + |\nabla u|^q \bigr) \bigl( (p-1)|\nabla u|^{p-2} + (q-1)|\nabla u|^{q-2} \bigr) u_{11} \\
	&- \left(1+\mu - \frac{1}{n}\right) u^{\beta+2\alpha}  + \bigl( \mu(\beta+\alpha) + \beta  \bigr) u^{\alpha+\beta-1} \bigl( |\nabla u|^p + |\nabla u|^q \bigr)  \\
	& + \lambda (\beta-1) u^{\beta-2} \bigl( |\nabla u|^p + |\nabla u|^q \bigr)\bigl( |\nabla u|^q + \frac{q(p-1)}{p(q-1)}|\nabla u|^p \bigr) \\
    &-\lambda u^{\alpha+\beta-1}\bigl( |\nabla u|^q + \frac{q(p-1)}{p(q-1)}|\nabla u|^p \bigr)\\
	& + u^\beta \bigl( |\nabla u|^{p-2} + |\nabla u|^{q-2} \bigr)^2 \sum_{i,j \ge 2} \left( u_{ij} - \frac{1}{n-1} \sum_{k \ge 2} u_{kk} \delta_{ij} \right)^2 \\
	& + 2u^\beta \bigl( |\nabla u|^{p-2} + |\nabla u|^{q-2} \bigr) \bigl( (p-1)|\nabla u|^{p-2} + (q-1)|\nabla u|^{q-2} \bigr) \sum_{i \ge 2} u_{1i}^2. 
	\end{split}
\end{align}

\subsection{Sign Control and Integral Estimates}

Let the first six terms on the right-hand side of \eqref{eq-7} be denoted by $B$. We set $y = \bigl[(p-1)|\nabla u|^{p-2} + (q-1)|\nabla u|^{q-2}\bigr] u_{11} + \frac{1}{n} u^\alpha$.

Recall that for $s=|\nabla u|$, we define
$ x = s^p + s^q,\,
z = (p-1)s^p + (q-1)s^q$, $X_q(s)=s^q+\frac{q(p-1)}{p(q-1)}s^p$, $r=\frac{X_q(s)}{x}\geq1,\, h_q=\frac{q}{q-1}$, then we have
\begin{align}\label{eq3.12}
	\begin{split}
	u^{-\beta} B
	&= \frac{n}{n-1} y^2 + (\beta+h_q\lambda) \left( y - \frac{1}{n} u^\alpha \right) \frac{x}{u} - \left( \mu + \frac{n-1}{n} \right) u^{2\alpha} \\
	& \quad + \bigl( \mu(\beta+\alpha) + \beta \bigr) u^{\alpha-1} x + \lambda(\beta-1)r \frac{x^2}{u^2}-\lambda r u^{\alpha-1} x\\
	&= \frac{n}{n-1} \left\{ y + \frac{n-1}{2n} \left(\beta+h_q\lambda\right) \frac{x}{u} \right\}^2 - \frac{n-1}{4n} \left( \beta+h_q\lambda \right)^2 \frac{x^2}{u^2} \\
	& + \lambda(\beta-1)r \frac{x^2}{u^2} - ( \mu + \frac{n-1}{n} ) u^{2\alpha} + \left[ \mu(\beta+\alpha) + \frac{n-1}{n} \beta - \frac{h_q}{n} \lambda - \lambda r \right] u^{\alpha-1} x \\
	&\geq Q_q\frac{x^2}{u^2} - \left( \mu + \frac{n-1}{n} \right) u^{2\alpha} + A_q u^{\alpha-1} x,
	\end{split}
\end{align}
where \begin{equation}\label{corrected-QA}
	Q_q=\lambda(\beta-1)r-\frac{n-1}{4n}(\beta+\lambda h_q)^2,
	\qquad
	A_q=\mu(\beta+\alpha)+\frac{n-1}{n}\beta-\lambda\left(r+\frac{h_q}{n}\right).
\end{equation}

Denote $$E^i_j=Y^i_j+\frac{1}{n}u^\alpha\delta_{ij}+ \frac{n-1}{2n} \left(\beta+h_q\lambda\right) \frac{|\nabla u|^{p-2} + |\nabla u|^{q-2}}{u}u_i u_j.$$ Then, \eqref{eq3.12} and \eqref{eq-7} implies
\begin{align}\label{corrected-sign}
	&\partial_i \left[ u^\beta Y_j^iY^j-u^\beta Y_j^jY^i+\lambda u^{\beta-1}X_qY^i+\mu u^{\beta+\alpha}Y^i \right]\non\\
	\geq& Q_q u^{\beta-2}x^2-\left(\mu+\frac{n-1}{n}\right)u^{\beta+2\alpha}
	 +A_q u^{\beta+\alpha-1}x+c_0 u^\beta |E|^2,
\end{align}
where \(c_0>0\) depends only on \(p,q,n\).

Choose
\[
	\lambda=-L,\qquad
	\mu=-\frac{n-1}{n}-\varepsilon_0,
\]
where \(L>0\) and \(0<\varepsilon_0\ll1\).
Since \(r\geq1\),
\[
	A_q\geq -\frac{n-1}{n}\alpha+L\left(1+\frac{h_q}{n}\right)-\varepsilon_0(\beta+\alpha),
\]
so \(A_q>0\), for sufficiently small \(\varepsilon_0\), if
\[
	L>\frac{(n-1)\alpha}{n+h_q}.
\]
Moreover, if $\beta\le 1$, then 
\[
	Q_q\geq L(1-\beta)-\frac{n-1}{4n}\left(\beta-h_q L\right)^2.
\]
Note that $Q_q$ is a quadratic polynomial in $\beta$, whose axis is $\beta=(h_q-\frac{2n}{n-1})L$. If 
we choose $\beta=(h_q-\frac{2n}{n-1})L$,
then \(Q_q\ge L\left[1-L\left(\frac{q}{q-1}-\frac{n}{n-1}\right)\right]>0\) if
\[
	L<\frac{(q-1)(n-1)}{n-q}.
\]
Note that when $L\in(\frac{(n-1)\alpha}{n+h_q},\frac{(q-1)(n-1)}{n-q})$, we have \begin{align*}
    &\beta\le \frac{(q-1)(n-1)}{n-q}(\frac{q}{q-1}-\frac{2 n}{n-1})=\frac{2 n-q n-q}{n-q}<1\quad\textit{if}\,q\in (1,\frac{2n}{n+1}],\\
    &\beta=(\frac{q}{q-1}-\frac{2 n}{n-1})L\le 0\quad\textit{if}\,q\in (\frac{2n}{n+1},n).
\end{align*}
Therefore, we can choose $\beta=(h_q-\frac{2n}{n-1})L$ and
these two requirements on $L$ are compatible exactly under
\[
	\alpha<\frac{q-1}{n-q}\left(n+\frac{q}{q-1}\right)
	=\frac{n(q-1)+q}{n-q}=q^*-1.
\]
Under the assumptions of Theorem \ref{Them1}, we may therefore choose \(L\) and then \(\varepsilon_0\) so that \(Q_q>0\), \(A_q>0\), and \(-(\mu+\frac{n-1}{n})>0\). Hence by \eqref{corrected-sign}, there exists \(\delta>0\) such that in $\O_r$
\begin{align}\label{ineq-id-new}
	\delta u^\beta\left\{|D Y|^2+u^{2\alpha}+\frac{x^2}{u^2}\right\}
	\leq
	\partial_i\left\{u^\beta Y_j^iY^j-u^\beta Y_j^jY^i+\lambda u^{\beta-1}X_qY^i+\mu u^{\beta+\alpha}Y^i\right\}.
\end{align}

The standard regularization argument for the p-Laplacian type equations implies that the inequality \eqref{ineq-id-new} holds globally in weak sense (see e.g. \cite{CM18,CFR20,Zhou24,YuZhou25}).

Let \(\eta\in C_c^\infty(B_{2R})\) satisfy \(\eta\equiv1\) in \(B_R\), \(0\leqslant\eta\leqslant1\), and \(|\nabla\eta|\leqslant C/R\). Testing \eqref{ineq-id-new} with \(\eta^\gamma\), where \(\gamma>2p\) is fixed, yields
\begin{align}\label{ineq-0-new}
	\begin{split}
		&\delta\int u^\beta\left\{|D Y|^2+u^{2\alpha}+\frac{x^2}{u^2}\right\}\eta^\gamma\,dx\\
		&\leq C\int u^\beta |DY|\,|Y|\,|\nabla\eta|\eta^{\gamma-1}\,dx
		+C\int u^{\beta-1}X_q|Y|\,|\nabla\eta|\eta^{\gamma-1}\,dx
		+C\int u^{\beta+\alpha}|Y|\,|\nabla\eta|\eta^{\gamma-1}\,dx.
	\end{split}
\end{align}
By Cauchy's inequality and \(X_q\asymp x\),
\begin{align}
	C\int u^\beta |DY|\,|Y|\,|\nabla\eta|\eta^{\gamma-1}\,dx
	&\leq \varepsilon\int u^\beta |D Y|^2\eta^\gamma\,dx
	+C\int u^\beta|Y|^2|\nabla\eta|^2\eta^{\gamma-2}\,dx,\\
	C\int u^{\beta-1}X_q|Y|\,|\nabla\eta|\eta^{\gamma-1}\,dx
	&\leq \varepsilon\int u^\beta\frac{x^2}{u^2}\eta^\gamma\,dx
	+C\int u^\beta|Y|^2|\nabla\eta|^2\eta^{\gamma-2}\,dx,\\
	C\int u^{\beta+\alpha}|Y|\,|\nabla\eta|\eta^{\gamma-1}\,dx
	&\leq \varepsilon\int u^{\beta+2\alpha}\eta^\gamma\,dx
	+C\int u^\beta|Y|^2|\nabla\eta|^2\eta^{\gamma-2}\,dx.
\end{align}
Taking \(\varepsilon>0\) small gives
\begin{align}\label{ineq-5-new}
	\int u^\beta\left\{|D Y|^2+u^{2\alpha}+\frac{x^2}{u^2}\right\}\eta^\gamma\,dx
	\leq C\int u^\beta|Y|^2|\nabla\eta|^2\eta^{\gamma-2}\,dx.
\end{align}

Note that \(|Y|^2\leq C(s^{2(p-1)}+s^{2(q-1)})\). Young's inequality gives, if $\beta+2(m-1)\ge 0$, then
\[
	R^{-2}\int_{B_{R}^c} u^\beta s^{2(m-1)}\eta^{\gamma-2}\,dx
	\leq \varepsilon\int u^\beta\left(\frac{s^{2m}}{u^2}+u^{2\alpha}\right)\eta^\gamma\,dx
	+CR^{\,n-\frac{m(\beta+2\alpha)}{\alpha-(m-1)}}.
\]
On the other hand, if $\beta+2(m-1)< 0$, then by Lemma \ref{lemma1}
\begin{align*}
	R^{-2}\int_{B_{R}^c} u^\beta s^{2(m-1)}\eta^{\gamma-2}\,dx
	&\leq \varepsilon\int u^{\beta-2}s^{2m}\eta^\gamma+CR^{-2m}\int_{B_{R}^c} u^{\beta+2(m-1)}\eta^{\gamma-2m}\,dx\\
&\le \varepsilon\int u^{\beta-2}s^{2m}\eta^\gamma+CR^{n-\frac{n-q}{q-1}(\beta+2m-2)-2m}
\end{align*}

\noindent
\textbf{Case 1:} $1<q\le\frac{2n}{n+1}$. Then we have $\beta+2(q-1)> 0$, since $\beta=(\frac{q}{q-1}-\frac{2n}{n-1})L\ge 0$.

Using estimates for \(m=p,q\) in \eqref{ineq-5-new}, and absorbing once more, we obtain
\begin{align}\label{final-estimate1}
	&\int u^\beta\left\{|D Y|^2+u^{2\alpha}+\frac{x^2}{u^2}\right\}\eta^\gamma\,dx
	\leq CR^{E_q}+CR^{E_p},\\
	&E_m=n-\frac{m(\beta+2\alpha)}{\alpha-(m-1)}.\non
\end{align}

\noindent
\textbf{Case 2:} $q>\frac{2n}{n+1}$.
Using estimates for \(m=p,q\) in \eqref{ineq-5-new}, and absorbing once more, we obtain
\begin{align}\label{final-estimate2}
	&\int u^\beta\left\{|D Y|^2+u^{2\alpha}+\frac{x^2}{u^2}\right\}\eta^\gamma\,dx
	\leq CR^{E_q}+CR^{E_p},\\
	&E_m=\max\{n-\frac{m(\beta+2\alpha)}{\alpha-(m-1)},n-\frac{n-q}{q-1}(\beta+2m-2)-2m\}.\non
\end{align}
It remains to verify that these exponents are negative. Put
\[
	L_0=\frac{(n-1)\alpha}{n+h_q}.
\]
At \(L=L_0\),
\[
	\beta+2\alpha=\frac{q(n+1)}{n(q-1)+q}\,\alpha,
\]
and therefore
\begin{align*}
    n-\frac{q(\beta+2\alpha)}{\alpha-(q-1)}
	=n-&\frac{q^2(n+1)\alpha}{(n(q-1)+q)(\alpha-q+1)}<0\,\,\textit{if} \,\alpha<q^*-1.\\
n-\frac{n-q}{q-1}(\beta+2q-2)-2q
=& -n - \frac{n-q}{q-1}\beta=-n - \frac{n-q}{q-1} \cdot \frac{-q(n+1)+2n}{n(q-1)+q}\alpha\\\
<&\frac{q-n}{q-1}<0\quad\textit{if}\,\,q> \frac{2n}{n+1},\alpha<q^*-1.
\end{align*}
 By continuity, choose \(L>L_0\) sufficiently close to \(L_0\) so that \(E_q<0\). Since \(p>q\) and \(\alpha>p-1\), the function
\[
	m\mapsto \frac{m}{\alpha-(m-1)}
\]
is increasing for \(m<\alpha+1\), and hence $E_m$ is decreasing for \(m<\alpha+1\). Therefore, we have $E_p<E_q<0$.

Letting \(R\to\infty\) in \eqref{final-estimate1} or \eqref{final-estimate2}, the right-hand side tends to zero. Since the integrand on the left is nonnegative and \(u>0\), this forces \(u^{\beta+2\alpha}\equiv0\), a contradiction. Theorem \ref{Them1} follows.

\section*{Acknowledgments}
The authors would like to express sincere gratitude to their supervisor, Professor Xi-Nan Ma, for his constant encouragement and guidance in this research. Yang Zhou is supported by National Natural Science Foundation of China [grant number 12141105]. Hua Zhu is supported by the National Natural Science Foundation of China (Grant No. 12501273) and the Research Foundation of Southwest University of Science and Technology (Grant No. 25zx7153).

\end{document}